\documentclass{amsart}
\usepackage{amssymb}
\usepackage{amsmath}
\usepackage{mathrsfs}
\usepackage[all]{xy}

\renewenvironment{proof}[1][Proof]{
 \topsep=0pt
 \partopsep=0pt
 \begin{trivlist}
  \itemindent=\parindent
  \item[\hskip \labelsep\emph{#1.}]
}{
 \qed
 \end{trivlist}
}

\newtheoremstyle{thmstyle}
  {2ex}  
  {0ex}  
  {\itshape}     
  {0ex}    
  {\bfseries}     
  {.}    
  {.5em} 
  {}     

\theoremstyle{thmstyle}
\newtheorem{theorem}{Theorem}[section]
\newtheorem{lemma}[theorem]{Lemma}
\newtheorem{corollary}[theorem]{Corollary}

\newtheorem{proposition}[theorem]{Proposition}

\renewenvironment{itemize}{%
\begin{list}{$\bullet$}{
 \setlength{\topsep}{0pt}
 \setlength{\partopsep}{2ex}
 \setlength{\itemsep}{0pt}
 \setlength{\labelwidth}{5ex}
}
}{\end{list}}

\renewenvironment{enumerate}{%
\begin{list}{(\arabic{enumi})}{
 \usecounter{enumi}
 \setlength{\topsep}{0pt}
 \setlength{\partopsep}{2ex}
 \setlength{\itemsep}{0pt}
 \setlength{\labelwidth}{5ex}
}
}{\end{list}}

\newcommand{\der}{\mathrm{der}}
\newcommand{\scn}{\mathrm{sc}}

\newcommand{\red}{\mathrm{red}}
\newcommand{\Frob}{\mathrm{Frob}}

\DeclareMathOperator{\End}{End}

\DeclareMathOperator{\Spec}{Spec}

\DeclareMathOperator{\ad}{ad}

\DeclareMathOperator{\Lie}{Lie}
\DeclareMathOperator{\Ext}{Ext}

\DeclareMathOperator{\Gal}{Gal}
\DeclareMathOperator{\chr}{char}
\DeclareMathOperator{\len}{len}
\DeclareMathOperator{\Rep}{Rep}

\DeclareMathOperator{\rk}{rk}
\DeclareMathOperator{\Res}{Res}

\DeclareMathOperator{\soc}{soc}

\def\GL{\mathrm{GL}}

\def\Z{\mathbf{Z}}
\def\G{\mathbf{G}}

\def\C{\mathbf{C}}

\def\Q{\mathbf{Q}}
\def\F{\mathbf{F}}
\def\A{\mathbf{A}}

\def\un{\mathrm{un}}
\def\wt{\mathbf{wt}}

\let\ms\mathscr
\let\mf\mathfrak
\let\mc\mathcal

\let\wt\widetilde
\let\ol\overline
\let\ul\underline

\title{Bigness in compatible systems}
\author{Andrew Snowden and Andrew Wiles}
\date{April 21, 2010}

\thanks{A.~Snowden was partially supported by NSF fellowship DMS-0902661.  A.~Wiles was supported by an NSF grant.}

\begin{document}

\begin{abstract}
Clozel, Harris and Taylor have recently proved a modularity lifting theorem of the following general form:  if $\rho$
is an $\ell$-adic representation of the absolute Galois group of a number field for which the residual representation
$\ol{\rho}$ comes from a modular form then so does $\rho$.  This theorem has numerous hypotheses; a crucial one is that
the image of $\ol{\rho}$ must be ``big,'' a technical condition on subgroups of $\GL_n$.  In this paper we investigate
this condition in compatible systems.  Our main result is that in a sufficiently irreducible compatible system the
residual images are big at a density one set of primes.  This result should make some of the work of Clozel, Harris and
Taylor easier to apply in the setting of compatible systems.
\end{abstract}

\maketitle
\setlength{\parskip}{0ex}
\tableofcontents

\setlength{\parindent}{0ex}
\setlength{\parskip}{2ex}

\section{Introduction}

Let $k$ be a finite field of characteristic $\ell$, let $V$ be a finite dimensional vector space over $k$ and let $G$
be a subgroup of $\GL(V)$.  For an endomorphism $g$ of $V$ and an element $\alpha$ of $k$ we let $V_{g, \alpha}$ denote
the generalized eigenspace of $g$ with eigenvalue $\alpha$.  It is naturally a sub and a quotient of $V$.  Following
Clozel, Harris and Taylor (see \cite[Def.~2.5.1]{CHT}), we say that $G$ is \emph{big} if the following four
conditions hold:
\begin{itemize}
\item[(B1)] The group $G$ has no non-trivial quotient of $\ell$-power order.
\item[(B2)] The space $V$ is absolutely irreducible as a $G$-module.
\item[(B3)] We have $H^1(G, \ad^{\circ}{V})=0$.
\item[(B4)] For every irreducible $G$-submodule $W$ of $\ad{V}$ we can find $g \in G$, $\alpha \in k$ and $f \in W$
such that $V_{g, \alpha}$ is one dimensional and the composite
\begin{displaymath}
\xymatrix{
V_{g, \alpha} \ar@{^(->}[r] & V \ar[r]^f & V \ar@{->>}[r] & V_{g, \alpha} }
\end{displaymath}
is non-zero.
\end{itemize}

The bigness condition is important in the work of Clozel, Harris and Taylor \cite{CHT}.  They prove modularity lifting
theorems of the following general form:  if $\rho$ is an $\ell$-adic representation of the absolute Galois group $G_F$ of
a number field $F$ such that $\ol{\rho}$ comes from a modular form then so does $\rho$.  There are several hypotheses in
these theorems, but one crucial one is that the image of $\ol{\rho}$ must be big.  In this paper, we investigate the
bigness condition in compatible systems and show that it automatically holds at a density one set of primes, assuming
the system is sufficiently irreducible.  Thus the theorems of \cite{CHT} should become easier to apply in the setting
of compatible systems.  Precisely, our main theorem is the following:

\begin{theorem}
\label{mainthm}
Let $F$ be a number field.  Let $L$ be a set of prime numbers of Dirichlet density one and for each $\ell \in L$
let $\rho_{\ell} : G_F \to \GL_n(\Q_{\ell})$ be a continuous representation of $G_F$.  Assume that the $\rho_{\ell}$
form a compatible system and that each $\rho_{\ell}$ is absolutely irreducible when restricted to any open subgroup of
$G_F$.  Then there is a subset $L' \subset L$ of density one such that $\ol{\rho}_{\ell}(G_F)$ is a big subgroup
of $\GL_n(\F_{\ell})$ for all $\ell \in L'$.
\end{theorem}

Here $\ol{\rho}_{\ell}$ denotes the semi-simplified mod $\ell$ reduction of $\rho_{\ell}$.  For our definition of
``compatible system'' see \S \ref{s:compsys}.  We in fact prove a more general result, allowing for compatible systems
with coefficients in a number field and for $F$ to be a function field; see \S \ref{s:bigsys} for details.  We note
that if one only assumes that the representations $\rho_{\ell}$ are absolutely irreducible, rather than absolutely
irreducible on any open subgroup, then our arguments can be used to show that $\ol{\rho}_{\ell}(G_F)$ satisfies
(B1), (B2) and (B3) at a set of primes $\ell$ of density one.  In particular, if $\{\rho_{\ell}\}$ is a compatible
system of absolutely irreducible representations then $\ol{\rho}_{\ell}$ is absolutely irreducible for almost all
$\ell$.

\subsection{Outline of proof}

Broadly speaking, the proof of Theorem~\ref{mainthm} has three steps:
\begin{enumerate}
\item We first show that if $G/\F_{\ell}$ is a reductive group and $\rho:G \to \GL_n$ is an absolutely irreducible
algebraic representation of $G$ such that $\ell$ is large compared to $n$ and the weights appearing in $\rho$ then the
group $\rho(G(\F_{\ell}))$ is big.
\item Using this, we show that if $\rho:\Gamma \to \GL_n(\Q_{\ell})$ is an $\ell$-adic representation of a profinite
group such that (a) $\ell$ is large compared to $n$; (b) $\rho$ is absolutely irreducible when restricted to any open
subgroup of $\Gamma$; and (c) $\rho(\Gamma)$ is close to being a hyperspecial subgroup of its Zariski closure, then
$\ol{\rho}(\Gamma)$ is a big subgroup of $\GL_n(\F_{\ell})$.
\item Finally, we combine the above with results of Serre and Larsen on compatible systems to deduce
Theorem~\ref{mainthm}.
\end{enumerate}

\subsection{Examples}

We should point out that one can construct compatible systems which satisfy the hypotheses of the theorem.  Let $F$ be
a totally real number field (resp.\ imaginary CM field) and let $\pi$ be a cuspidal automorphic representation of
$\GL_n(\A_F)$ satisfying the following conditions:
\begin{itemize}
\item[(C1)] $\pi$ is \emph{regular algebraic}.  This means that $\pi_{\infty}$ has the same infinitesimal character as
some irreducible algebraic representation of the restriction of scalars from $F$ to $\Q$ of $\GL_n$.
\item[(C2)] $\pi$ is \emph{essentially self-dual} (resp.\ \emph{conjugate self-dual}).  When $F$ is totally real this
means that $\pi^{\vee}=\chi \otimes \pi$ for some character $\chi$ of the idele group of $F$ for which $\chi_v(-1)$ is
independent of $v$, as $v$ varies over the infinite places of $F$.  When $F$ is imaginary CM, this means that
$\pi^{\vee}=\pi^c$, where $c$ denotes complex conjugation.
\item[(C3)] There is some finite place $v_0$ of $F$ such that $\pi_{v_0}$ is a twist of the Steinberg representation.
\end{itemize}
Under these conditions, we can associate to $\pi$ a compatible system of semi-simple representations $\{\rho_w \}$ of
$G_F$ with coefficients in some number field $E$.  The system is indexed by the places $w$ of $E$.  For more precise
statements, see \cite[Prop.~3.2.1]{CHT} and \cite[Prop.~3.3.1]{CHT}.

Let $w$ be a place of $E$ with residue characteristic different from that of $v_0$.  Assume that $F$ is imaginary CM.
By the main result of \cite{TY} the (Frobenius semi-simplification of the) representation $\rho_w \vert_{G_{F, v_0}}$
corresponds to $\pi_{v_0}$ under the local Langlands correspondence (we write $G_{F, v_0}$ for the decomposition group
at $v_0$).  As $\pi_{v_0}$ is a twist of the Steinberg representation, we find that $\rho_w \vert_{G_{F, v_0}}$ is
absolutely indecomposable, and remains so after restricting to any open subgroup of $G_{F, v_0}$.  It follows that
$\rho_w$ is absolutely indecomposable when restricted to any open subgroup of $G_F$.  Since $\rho_w$ is semi-simple, we
conclude that it is in fact absolutely irreducible when restricted to any open subgroup of $G_F$.  When $F$ is totally
real we can still conclude that $\rho_w$ has this property by making an appropriate abelian base change to an imaginary
CM field and appealing to the above argument.

We thus see that all but finitely many members of the compatible system $\{ \rho_w \}$ are absolutely irreducible on
any open subgroup of $G_F$.  By the more general version of the main theorem given in \S \ref{s:bigsys}, we conclude
that there is a set of primes $P$ of $\Q$ of Dirichlet density $1/[E:\Q]$, all of which split in $E$, such that
$\ol{\rho}_w(G_F)$ is big subgroup of $\GL_n(\F_{\ell})$ for all $w$ which lie above a prime $\ell \in P$.

\subsection{Notation and conventions}

Reductive groups over fields are connected.  A semi-simple group $G$ over a field $k$ is called simply connected if
the root datum of $G_{\ol{k}}$ is simply connected (i.e., coroots span the coweight lattice).  If $G$ is a semi-simple
group over $k$ then there is a simply connected group $G^{\scn}$ and an isogeny $G^{\scn} \to G$ whose kernel is
central.  The group $G^{\scn}$ and the map $G^{\scn} \to G$ are unique up to isomorphism.  We call $G^{\scn}$ the
universal cover of $G$.  For an arbitrary algebraic group $G$ over $k$ we let $G^{\circ}$ denote the connected component
of the identity, $G^{\ad}$ the adjoint group of the quotient of $G^{\circ}$ by its radical, which is a semi-simple
group, and $G^{\scn}$ the universal cover of $G^{\ad}$.  We also write $G^{\der}$ for the derived subgroup of
$G^{\circ}$, which is semi-simple if $G^{\circ}$ is reductive.  For a vector space $V$ we denote by $\ad{V}$ the space
of endomorphisms of $V$ and by $\ad^{\circ}{V}$ the subspace of traceless endomorphisms.  More definitions are given in
the body of the paper.

\subsection*{Acknowledgments}

We would like to thank Thomas Barnet-Lamb, Bhargav Bhatt, Brian Conrad, Alireza Salehi Golsefidy and Jiu-Kang Yu for
useful discussions.  We would also like to thank an anonymous referee for some helpful comments.

\section{Elementary properties of bigness}
\label{s:elem}

In this section we establish some elementary properties of bigness.  Throughout this section, $k$ denotes a finite
field of characteristic $\ell$, $V$ a finite dimensional vector space over $k$ and $G$ a subgroup of $\GL(V)$.

\begin{proposition}
\label{elem-1}
Let $H$ be a normal subgroup of $G$.  If $H$ satisfies (B2), (B3) and (B4) then $G$ does as well.  In particular, if
$H$ is big and the index $[G:H]$ is prime to $\ell$ then $G$ is big.
\end{proposition}

\begin{proof}
Assume $H$ satisfies (B2), (B3) and (B4).  Since $V$ is absolutely irreducible for $H$, it is for $G$ as well, and so
$G$ satisfies (B2).  We have an exact sequence
\begin{displaymath}
H^1(G/H, (\ad^{\circ}{V})^H) \to H^1(G, \ad^{\circ}{V}) \to H^1(H, \ad^{\circ}{V})^{G/H}.
\end{displaymath}
Since $H$ satisfies (B3), $H^1(H, \ad^{\circ}{V})=0$ and so the rightmost term vanishes.  Since $V$ is absolutely
irreducible for $H$ we have $(\ad^{\circ} {V})^H=0$ and so the leftmost term vanishes.  Thus $H^1(G, \ad^{\circ}{V})=0$
and $G$ satisfies (B3).  Now let $W$ be a $G$-irreducible submodule of $\ad{V}$.  Let $W'$ be an $H$-irreducible
submodule of $W$.  Since $H$ satisfies (B4), we can find $g \in H$, $\alpha \in k$ and $f \in W'$ such that
$V_{g, \alpha}$ is one dimensional and $f(V_{g, \alpha})$ has non-zero projection to $V_{g, \alpha}$.  Of course, $g$
also belongs to $G$ and $f$ also belongs to $W$.  Thus $G$ satisfies (B4) as well.

Now say that $H$ is big and $[G:H]$ is prime to $\ell$.  The above arguments show that $G$ satisfies (B2), (B3) and
(B4), so to show that $G$ is big we need only verify (B1).  Let $K$ be an $\ell$-power order quotient of $G$.  Since
$H$ has no $\ell$-power order quotient, its image in $K$ is trivial.  Thus $K$ is a quotient of $G/H$.  But this
group has prime-to-$\ell$ order, and so $K=1$.  This shows that the only $\ell$-power order quotient of $G$ is the
trivial group, and so $G$ satisfies (B1).
\end{proof}

\begin{proposition}
\label{elem-2}
The group $G$ is big if and only if $k^{\times} G$ is, where $k^{\times}$ denotes the group of scalar matrices in
$\GL(V)$.
\end{proposition}

\begin{proof}
Since $G$ is a normal subgroup of $k^{\times} G$ of prime-to-$\ell$ index, the bigness of the former implies that of the
latter by Proposition~\ref{elem-1}.  Now assume that $k^{\times} G$ is big.  Let $K$ be an $\ell$-power order quotient
of $G$.  Since $k^{\times} \cap G$ is prime to $\ell$, its image in $K$ is trivial.  Thus $K$ is a quotient of the
group $G/(G \cap k^{\times})=k^{\times} G/k^{\times}$.  By assumption, $k^{\times} G$ has no non-trivial quotient of
$\ell$-power order.  Thus $K=1$ and $G$ satisfies (B1).

Since $V$ is absolutely irreducible for $k^{\times} G$ it is for $G$ as well.  Thus $G$ satisfies (B2).

We have an exact sequence
\begin{displaymath}
1 \to G \to k^{\times} G \to H \to 1
\end{displaymath}
where $H$ is a quotient of $k^{\times}$.  We thus have an exact sequence
\begin{displaymath}
H^1(k^{\times} G, \ad^{\circ}{V}) \to H^1(G, \ad^{\circ}{V})^H \to H^2(H, (\ad^{\circ}{V})^G).
\end{displaymath}
The group on the left vanishes by hypothesis.  The group on the right vanishes since $(\ad^{\circ}{V})^G=0$.  Thus the
group in the middle vanishes.  Now, the action of $H$ on $H^1(G, \ad{^{\circ}V})$ is trivial.  (Proof:  Let $f:G \to
\ad^{\circ}{V}$ be a 1-cocycle representing representing a cohomology class $[f]$ and let $h$ be an element of $H$.
Then $h \cdot [f]$ is represented by the 1-cocycle $g \mapsto \wt{h} f(\wt{h}^{-1} g \wt{h})$ for any lift $\wt{h}$ of
$h$.  We can pick a lift $\wt{h}$ of $h$ which belongs to $k^{\times}$.  Thus $\wt{h}$ acts trivially on $G$ by
conjugation and acts trivially on $\ad^{\circ}{V}$.  Therefore $h \cdot [f]=[f]$.)  It thus follows that
$H^1(G, \ad^{\circ}{V})$ vanishes and so $G$ satisfies (B3).

As for the last condition, let $W$ be an irreducible $G$-submodule of $\ad{V}$.  Then it is also an irreducible
$k^{\times} G$-module.  Thus we can find $g \in k^{\times} G$, $\alpha \in k$ and $f \in W$ such that $V_{g, \alpha}$
is one dimensional and $f(V_{g, \alpha})$ has non-zero projection to $V_{g, \alpha}$.  We can write $g=z g'$ where $z$
belongs to $k^{\times}$ and $g'$ belongs to $G$.  Put $\alpha'=\alpha z^{-1}$.  Then $V_{g', \alpha'}=V_{g, \alpha}$.
Thus this space is one dimensional and $f(V_{g', \alpha'})$ has non-zero projection to $V_{g', \alpha'}$.  We have
therefore shown that $G$ satisfies (B4).  Thus $G$ is big.
\end{proof}

The following result will not be used, but is good to know.

\begin{proposition}
\label{elem-3}
Let $k'/k$ be a finite extension and put $V'=V \otimes k'$.  If $G$ is a big subgroup of $\GL(V)$ then it is a big
subgroup of $\GL(V')$.
\end{proposition}

\begin{proof}
Conditions (B1), (B2) and (B3) for $G$ as a subgroup of $\GL(V')$ are immediate.  We prove (B4).  Let $S$ be the set of
pairs $(g, \alpha) \in G \times k^{\times}$ such that $V_{g, \alpha}$ is one dimensional.  Consider the natural map
\begin{displaymath}
\Phi : \ad{V} \to \bigoplus_{(g, \alpha) \in S} \End(V_{g, \alpha}).
\end{displaymath}
The map $\Phi$ is $G$-equivariant, using the natural action of $G$ on the target.  Let $W$ be an irreducible submodule
of $\ad{V}$.  Since $G \subset \GL(V)$ satisfies (B4), we can find $f \in W$ such that the image of $f$ in $\End(V_{g,
\alpha})$ is non-zero, for some $(g, \alpha) \in S$.  Clearly then, $\Phi(f) \ne 0$.  It follows that $\Phi(W) \ne 0$,
and therefore (since $W$ is irreducible), $\Phi \vert_W$ is injective.  As this holds for every irreducible submodule
of $\ad{V}$, it follows that $\Phi$ is injective.  Tensoring $\Phi$ with $k'$, we find that the natural map
\begin{displaymath}
\Phi' : \ad{V'} \to \bigoplus_{(g, \alpha) \in S} \End(V'_{g, \alpha})
\end{displaymath}
is injective.  (Note that $V_{g, \alpha} \otimes_k k'=V'_{g, \alpha}$.)  Let $W$ be an irreducible submodule of
$\ad{V'}$.  Given any non-zero $f \in W$ the image of $f$ under $\Phi'$ is non-zero.  It follows that there exists
some $(g, \alpha) \in S$ such that the image of $f$ in $\End(V'_{g, \alpha})$ is non-zero.  This shows that
$G \subset \GL(V')$ satisfies (B4).
\end{proof}

\section{Background on representations of algebraic groups}

In this section we review some representation theory of algebraic groups.  To a representation $V$ of an algebraic group
$G$ we attach a non-negative integer $\Vert V \Vert$ which measures the size of the weights appearing in $V$.  The
key principle we need is:  over a field of positive characteristic, the representations of $G$ with $\Vert V \Vert$
small behave in many ways like representations in characteristic 0.  We will give several precise statements of this
type.

\subsection{Borel-Weil type representations}

Let $S$ be a scheme.  A group scheme $G/S$ is \emph{reductive} (resp.\ \emph{semi-simple}) if it is smooth, affine and
its geometric fibers are reductive (resp.\ semi-simple).  This implies that its geometric fibers are connected, by our
conventions.  Such a group is a \emph{torus} if it is fppf locally isomorphic to $\G_m^n$.  A torus is \emph{split} if
it is (globally) isomorphic to $\G_m^n$ (it may be best to allow $n$ to be a locally constant function on
the base $S$; this will not be an issue for us).  By a \emph{maximal torus} in $G$ we mean a subtorus which is maximal
in each geometric fiber.  Similarly, by a \emph{Borel subgroup} we mean a closed subgroup which is smooth over $S$ and
a Borel subgroup in each geometric fiber.  A reductive group $G/S$ is \emph{split} if it has a split maximal torus $T$
such that the weight spaces of $T$ on $\Lie(G)$ are free coherent sheaves on $S$.  See \cite[Exp.~XIX]{SGA} for the
general theory.

Let $G/S$ be a split reductive group over a connected locally noetherian base $S$.  Let $B$ be a Borel subgroup of $G$
and let $T \subset B$ be a split maximal
torus.  Let $\lambda$ be a dominant weight of $T$ and let $\mc{L}_S(\lambda)$ be the natural $G$-equivariant line bundle
on $G/B$ associated to $\lambda$. Put $V_{S, \lambda}=f_* \mc{L}_S(\lambda)$, where $f:G/B \to S$ is the structure
map.  Thus $V_{S, \lambda}$ is a coherent sheaf on $S$ with a natural action of $G$.  We call these sheaves
``Borel-Weil representations.'' We omit the $S$ from the notation if it is clear from context.  Consider a cartesian
diagram
\begin{displaymath}
\xymatrix{
(G/B)_{S'} \ar[r]^{g'} \ar[d]_{f'} & G/B \ar[d]^f \\ S' \ar[r]^g & S. }
\end{displaymath}
Note that formation of $\mc{L}_S(\lambda)$ commutes with pull-back, that is, $(g')^* \mc{L}_S(\lambda)=\mc{L}_{S'}
(\lambda)$.  Kempf's vanishing theorem \cite[Prop.~II.4.5]{Jan} states that if $S'$ is a geometric point of $S$
then $R^i f'_* \mc{L}_{S'}(\lambda)=0$ for $i>0$.  (When $S'$ has characteristic 0 this is part of the classical
Borel-Weil-Bott theorem.)  Thus, using a combination of the formal functions theorem and the proper base change theorem
(see also the chapter ``Cohomology and base change'' in \cite{Mum}), we see that $V_{S, \lambda}$ is a locally free
sheaf on $S$ and its formation commutes with base change, that is, for any diagram as above we have $V_{S', \lambda}
=g^* V_{S, \lambda}$.

Assume now that $S=\Spec(k)$ with $k$ an algebraically closed field.  If $k$ has characteristic zero then $V_{\lambda}$
is an irreducible representation of $G$.  Furthermore, every irreducible representation of $G$ is isomorphic to a unique
$V_{\lambda}$.  This is the classical Borel-Weil theorem.  If $k$ does not have characteristic zero then $V_{\lambda}$
may not be irreducible.  However, it has a unique irreducible submodule $\soc(V_{\lambda})$ and any irreducible
representation of $G$ is isomorphic to a unique $\soc(V_{\lambda})$ (see \cite[Cor.~II.2.7]{Jan}).  The
representation $\soc(V_{\lambda})$ is the unique irreducible with $\lambda$ as its highest weight.

\subsection{The norm of a representation}

Let $G/k$ be a reductive group over a field $k$.  Assume for the moment that $G$ is split and pick a split maximal
torus $T$ of $G$.  For a weight $\lambda$ of $T$ we let $\Vert \lambda \Vert$ be the maximum value of $\vert \langle
\lambda, \alpha^{\vee} \rangle \vert$ as $\alpha$ varies over the roots of $G$ with respect to $T$.  Let $V$ be a
representation of $G$.  We let $\Vert V \Vert$ be the maximum value of the $\Vert \lambda \Vert$ among the weights
$\lambda$ appearing in $V$.  The value of $\Vert V \Vert$ is independent of the choice of the torus $T$; furthermore,
if $k'/k$ is a field extension then $\Vert V_{k'} \Vert=\Vert V \Vert$.  Now drop the assumption that $G$ is split.
For a representation $V$ of $G$ we define $\Vert V \Vert$ to be $\Vert V_{k'} \Vert$ where $k'/k$ is an extension over
which $G$ splits.  It is clear that if $V$ is an extension of $V_1$ by $V_2$ then $\Vert V \Vert
=\max(\Vert V_1 \Vert, \Vert V_2 \Vert)$.  We also have the following:

\begin{proposition}
\label{rep-bg-0}
Let $f:G' \to G$ be a map of reductive groups over a field $k$ and let $V$ be a representation of $G$.  Assume one of
the following holds:
\begin{enumerate}
\item $f$ is a central isogeny.
\item $f$ is the projection onto a direct factor.
\item $f$ is the inclusion of the derived subgroup $G'$ of $G$.
\item $f$ is a surjection and $\ker{f}/(Z' \cap \ker{f})$ is smooth, where $Z'$ is the center of $G'$.
\end{enumerate}
Then $\Vert f^*V \Vert=\Vert V \Vert$.  (Note (4) subsumes (1) and (2).)
\end{proposition}

\begin{proof}
We may prove the proposition after passing to the closure of $k$.  We thus assume $k$ that is closed, and therefore,
that $G$ and $G'$ are split.  Although (4) subsumes (1) and (2) we will use (1) and (2) in the proof of (4), and so
prove them separately.

(1) Let $T$ be a split maximal torus of $G$.  Then $T'=f^{-1}(T)$ is a split maximal torus of $G'$.  Every weight of
$f^* V$ is of the form $f^* \lambda$ where $\lambda$ is a weight of $V$.  For a coroot $\alpha^{\vee}$ of $G'$ we have
the identity $\langle f^* \lambda, \alpha^{\vee} \rangle=\langle \lambda, f_* \alpha^{\vee} \rangle$.  The push-forward
$f_* \alpha^{\vee}$ is a coroot of $G$, and every coroot arises in this manner.  We thus find $\Vert f^* V \Vert=
\Vert V \Vert$.

(2) Write $G'=G \times G''$ so that $f$ is the projection onto $G$.  Let $T$ be a split maximal torus
of $G$ and $T''$ a split maximal torus of $G''$ so that $T'=T \times T''$ is a split maximal torus of $G'$.  The
weights of $f^* V$ coincide with the weights of $V$ in the obvious manner.  The coroots of $G'$ are the union of the
coroots of $G$ and $G''$.  As any coroot of $G''$ pairs to zero with a weight of $T$, we find $\Vert f^* V \Vert=
\Vert V \Vert$.

(3) Let $G'=G^{\der}$ and let $f:G' \to G$ be the natural inclusion.  Let $T$ be a split maximal torus of $G$.  Then
the reduced subscheme of the connected component of the identity of $f^{-1}(T)$ is a group (by the lemma following this
proof) and is a split maximal torus $T'$ of $G'$.
If $\alpha^{\vee}$ is a coroot of $G'$ then $f_*\alpha^{\vee}$ is a coroot of $G$ and all coroots arise in this manner.
We also have an adjointness between $f^*$ on weights and $f_*$ on coweights.  The proof now proceeds as in part (1).

(4) We have a diagram
\begin{displaymath}
\xymatrix{
\wt{G}' \ar[r]^{p'} \ar[d]_{\wt{f}} & (G')^{\der} \ar[r]^{i'} \ar[d]^{f'} & G' \ar[d]^f \\
\wt{G} \ar[r]^p & G^{\der} \ar[r]^i & G }
\end{displaymath}
Here $\wt{G}$ is the universal cover of $G^{\der}$ and similarly for $\wt{G}'$.  The map $\wt{f}$ is a lift of $f'$.
Let $H=\ker{\wt{f}}$ and let $H_{\red}$ be the reduced subscheme of $H$; it is a normal subgroup of $\wt{G}'$ by the
lemma following this proof.  Of course, $H_{\red}$ is smooth since it is reduced.  Let $K=H \cap \wt{Z}'$.  Then
$H/K=\ker{f}/(\ker{f} \cap Z')$ is smooth by hypothesis.  The map $H_{\red} \to H/K$ is between smooth groups
of the same dimension and is surjective on connected components; it is therefore surjective.  We thus have $H=K
H_{\red}$.  The map $\wt{f}$ can now be factored as
\begin{displaymath}
\wt{G}' \to \wt{G}'/H_{\red}^{\circ} \to \wt{G}'/H_{\red} \to \wt{G}'/H = \wt{G}
\end{displaymath}
The kernel of the first map is $H_{\red}^{\circ}$, which is a direct factor of $\wt{G}'$ since it is smooth, connected
and normal.  The kernel of the second map is $\pi_0(H_{\red})$, which is \'etale and therefore central.  The kernel
of the third map is $H/H_{\red}$, which is the image of $K$ in $\wt{G}'/H_{\red}$, and therefore central.  We thus see
that $\wt{f}$ is a composition of a projection onto a direct factor with two central isogenies.  It follows from (1) and
(2) that $\Vert \wt{f}^* W \Vert=\Vert W \Vert$ for any representation $W$ of $\wt{G}$.  We now have:
\begin{displaymath}
\Vert f^* V \Vert=\Vert (i')^* f^* V \Vert=\Vert (p')^* (i')^* f^* V \Vert=\Vert \wt{f}^* p^* i^* V \Vert=
\Vert p^* i^* V \Vert=\Vert i^* V \Vert=\Vert V \Vert.
\end{displaymath}
The first equality uses (3), the second (1), the third the commutativity of the diagram, the fourth the fact that
$\wt{f}^*$ preserves norm, the fifth (1), the sixth (3).
\end{proof}

\begin{lemma}
\label{rep-bg-8}
Let $G$ be an affine group over a field $k$ and let $G_{\red}$ be the reduced subscheme of $G$.
\begin{enumerate}
\item If $k$ is perfect then $G_{\red}$ is a subgroup of $G$.
\item If $G$ is a closed normal subgroup of a smooth affine group $H$ then $G_{\red}$ is stable under conjugation by
$H$.
\end{enumerate}
Thus if $k$ is perfect and $G$ is a closed normal subgroup of a smooth affine group then $G_{\red}$ is a closed normal
subgroup of $G$.
\end{lemma}

\begin{proof}
(1) Since $k$ is perfect, the product $G_{\red} \times G_{\red}$ (fiber product over $k$) is reduced.  Therefore the
composite
\begin{displaymath}
G_{\red} \times G_{\red} \to G \times G \to G
\end{displaymath}
factors through the inclusion $G_{\red} \to G$.  This shows that $G_{\red}$ is a subgroup of $G$.

(2) Given $h \in H(k)$, the map $G \to G$ given by conjugation by $h$ induces a map $G_{\red} \to G_{\red}$.  If
$k$ is infinite then $H(k)$ is dense in $H$, and so $G_{\red}$ is stable under conjugation by $H$.
If $k$ is finite then it is perfect, and one may therefore verify that $G_{\red}$ is stable by conjugation
after passing to the closure; since the closure is infinite, the previous argument applies.
\end{proof}

We thank Brian Conrad for informing us of counterexamples to the above statements when the hypotheses are not in
place.

\subsection{Representations of small norm}

We let $\Rep(G)$ be the category of representations of $G$.  For an integer $n$ we let $\Rep^{(n)}(G)$ be the full
subcategory of $\Rep(G)$ on those representations $V$ which satisfy $\Vert V \Vert < n$.  Both $\Rep(G)$ and
$\Rep^{(n)}(G)$ are abelian categories.  Furthermore, if
\begin{displaymath}
0 \to V' \to V \to V'' \to 0
\end{displaymath}
is an exact sequence in $\Rep(G)$ then $V$ belongs to $\Rep^{(n)}(G)$ if and only if both $V'$ and $V''$ do.  In other
words, $\Rep^{(n)}(G)$ is a Serre subcategory of $\Rep(G)$.

\begin{proposition}
\label{rep-bg-1}
Let $G/k$ be a reductive group over an algebraically closed field $k$.  Assume $\chr{k}$ is zero or large compared to
$n$ and $\dim{G}$.  Then:
\begin{enumerate}
\item The category $\Rep^{(n)}(G)$ is semi-simple.
\item The simple objects of $\Rep^{(n)}(G)$ are Borel-Weil representations.
\end{enumerate}
In other words, any representation $V$ of $G$ with $\Vert V \Vert$ small compared to $\chr{k}$ is a direct sum of
$V_{\lambda}$'s.
\end{proposition}

\begin{proof}
The statements are well-known in characteristic zero, so we assume $k$ has positive characteristic.  We prove (2) first.
The simple objects of $\Rep(G)$ are exactly the $\soc(V_{\lambda})$.  Now, $\Vert \soc(V_{\lambda})
\Vert \ge \Vert \lambda \Vert$ as $\lambda$ occurs as a weight in $\soc(V_{\lambda})$.  Thus if
$\soc(V_{\lambda})$ belongs to $\Rep^{(n)}(G)$ then $\Vert \lambda \Vert < n$.  On the other hand, it is
known that for $\chr{k}$ large compared to $\dim{G}$ and $\Vert \lambda \Vert$ the representation $V_{\lambda}$
is irreducible (see \cite{Spr}).  This proves (2).

We now prove (1).
First note that the simple objects of $\Rep^{(n)}(G)$ have dimension bounded in terms of $n$ and $\dim{G}$.  Indeed,
the group $G$ is the pull-back to $k$ of a unique split reductive group over $\Z$, which we still call $G$.  The
simple $V_{k, \lambda}$ is just $V_{\Z, \lambda} \otimes k$.  Thus the dimension of $V_{k, \lambda}$ is the same as the
dimension of $V_{\C, \lambda}$.  This dimension can then be bounded in terms of $\dim{G}$ and $\Vert \lambda
\Vert$ using the Weyl dimension formula \cite[Cor.~24.6]{FH} and the fact that there are only finitely
many root data of a given rank.

Now, it is known (see \cite{Jan2}, \cite{Lar2}) that any representation of $G$ with small dimension compared to
$\chr{k}$ is semi-simple.  Since $\chr{k}$ is large, we thus find that if $A$ and $B$ are two simples of $\Rep^{(n)}(G)$
then any extension of $A$ by $B$ is semi-simple, and therefore $\Ext^1(A, B)=0$.  This shows that $\Rep^{(n)}(G)$ is
semi-simple.  This completes the proof of the proposition.
\end{proof}

\begin{corollary}
\label{rep-bg-7}
Let $G/k$ be a reductive group over an algebraically closed field $k$ of characteristic $\ell$ and let $V$ be a
representation of $G$ such that $\dim{V}$ and $\Vert V \Vert$ are small compared to $\ell$.  Then $V$ is semi-simple
and a direct sum of simple Borel-Weil representations.
\end{corollary}

\begin{proof}
Let $H$ be the kernel of $G \to \GL(V)$, let $H'$ be the reduced subscheme of $H^{\circ}$ and let $G'=G/H'$.  Then
the map $G' \to \GL(V)$ has finite kernel.  Thus $\dim{G'}$ is bounded by $\dim{V}$ and is therefore small compared
to $\ell$.  By Proposition~\ref{rep-bg-0}, $\Vert V \Vert$ is the same for $G$ and $G'$.  By Proposition~\ref{rep-bg-1},
$V$ is semi-simple for $G'$ and a direct sum of simple Borel-Weil representations.  It follows that the same
holds for $G$.  (The restriction of a Borel-Weil representation along a surjection is still Borel-Weil.)
\end{proof}

\begin{proposition}
\label{rep-bg-4}
Let $K/\Q_{\ell}$ be an extension with ring of integers $\ms{O}_K$ and residue field $k$.  Let $G/\ms{O}_K$ be a
reductive group and let $V$ be a finite free $\ms{O}_K$-module with a representation of $G$.  Then $\Vert V_k \Vert=
\Vert V_K \Vert$ and this is bounded in terms of $\rk{V}$.  If the representation of $G_K$ on $V_K$ is absolutely
irreducible and $\ell$ is large compared to $\rk{V}$ then the representation of $G_k$ on $V_k$ is absolutely
irreducible.
\end{proposition}

\begin{proof}
By enlarging $K$ if necessary we can assume that $G$ is split.  Let $T$ be a split maximal torus in $G$.  As
maps of tori are rigid, the weights of $T$ in $V_K$ and $V_k$ are the same.  This shows that their norms agree.
The fact that $\Vert V_K \Vert$ is bounded in terms of $\rk{V}$ is a fact about representations of
complex Lie groups and can be proved using the Weyl dimension formula \cite[Cor.~24.6]{FH}.

Now, assume that $V_{\ol{K}}$ is irreducible for the action of $G_{\ol{K}}$ and that $\ell$ is large compared to
$\rk{V}$.  By the first paragraph, $\ell$ is large compared to $\Vert V_k \Vert$.  It thus follows from
Corollary~\ref{rep-bg-7} that we can write $V_{\ol{k}}=\bigoplus V_{\ol{k}, \lambda_i}$ with each $V_{\ol{k},
\lambda_i}$ irreducible.  The representations $V_{\ol{k}, \lambda_i}$ lift to $\ms{O}_{\ol{K}}$.  By the first
paragraph, we see that $V_{\ol{K}}$ and $\bigoplus V_{\ol{K}, \lambda_i}$ have the same weights, and are thus
isomorphic.  Since $V_{\ol{K}}$ is irreducible, there must therefore be only one term in the sum, and so $V_{\ol{k}}$
must be irreducible as well.
\end{proof}

\subsection{Representations of $G(k)$}

Let $k$ be a finite field and let $G/k$ be a reductive group.  We denote by $\Rep(G(k))$ the category of representations
of the finite group $G(k)$ on $\ol{k}$-vector spaces.

\begin{proposition}
\label{rep-bg-3}
Let $k$ be a finite field and $G/k$ a semi-simple simply connected group.  Assume $\chr{k}$ is sufficiently large
compared to $\dim{G}$ and $n$.  Then the functor $\Rep^{(n)}(G_{\ol{k}}) \to \Rep(G(k))$ is fully faithful and the
essential image is a Serre subcategory of $\Rep(G(k))$.
\end{proposition}

\begin{proof}
The functor $\Rep^{(n)}(G_{\ol{k}}) \to \Rep(G(k))$ is clearly faithful and exact.  The desired properties now follow
from the fact that $\Rep^{(n)} (G_{\ol{k}})$ is semi-simple and the fact that if $V$ is an irreducible representation
of $G_{\ol{k}}$ with norm small compared to $\chr{k}$ then it stays irreducible when restricted to $G(k)$
(see \cite[\S 1.13]{Lar}).
\end{proof}

\subsection{Representations of $\Lie(G)$}

We will need the following result:

\begin{proposition}
\label{rep-bg-2}
Let $k$ be the algebraic closure of a finite field, $G/k$ a semi-simple group and $V$ an irreducible representation of
$G$.  Pick a system of positive roots $P$ in $\mf{g}=\Lie(G)$.  Assume $\chr{k}$ is large compared to $\Vert V
\Vert$ and $\dim{G}$.  Then the subspace of $V$ annihilated by $P$ is one dimensional.  (This subspace is the
highest weight space of $V$ with respect to $P$.)
\end{proposition}

\begin{proof}
Denote by $G$ still the unique split group over $\Z$ giving rise to $G/k$.  By our hypotheses we have $V=V_{k, \lambda}$
for some dominant weight $\lambda$.  We know that $V_{k, \lambda}=V_{\Z, \lambda} \otimes k$ and similarly
$V_{\C, \lambda}=V_{\Z, \lambda} \otimes \C$.  Now, since $G$ is split over $\Z$ for each $r \in P$ we can find $X_r
\in \mf{g}_{\Z}$ which generates the $r$ root space of $\mf{g}_{\Z}$.  Consider the map
\begin{displaymath}
V_{\Z, \lambda} \to \bigoplus_P V_{\Z, \lambda}, \qquad v \mapsto (X_r \cdot v)_{r \in P}.
\end{displaymath}
This is a linear map of finite free $\Z$-modules.  After tensoring with $\C$ the kernel of this map is the subspace of
$V_{\C, \lambda}$ annihilated by $P$.  This is one dimensional by the usual highest weight theory over $\C$.  It thus
follows that for $\ell$ sufficiently large, the reduction of the map modulo $\ell$ will have one dimensional kernel.
This proves the proposition.
\end{proof}

\section{Highly regular elements of semi-simple groups}

Fix a finite field $k$ of cardinality $q$.  The purpose of this section is to demonstrate the following result:

\begin{proposition}
\label{reg-1}
Let $G/k$ be a semi-simple group and let $T \subset G$ be a maximal torus defined over $k$.  Let $n$ be an integer and
assume $q$ is large compared to $\dim{G}$ and $n$.  Then there exists an element $g \in T(k)$ for which the map
\begin{displaymath}
\{ \lambda \in X(T_{\ol{k}}) \mid \Vert \lambda \Vert <n \} \to \ol{k}^{\times},
\qquad \lambda \mapsto \lambda(g)
\end{displaymath}
is injective.
\end{proposition}

Before proving the proposition we give a few lemmas.

\begin{lemma}
\label{reg-2}
Let $T/k$ be a torus of rank $r$.  Then $(q-1)^r \le \# T(k) \le (q+1)^r$.
\end{lemma}

\begin{proof}
We have $\# T(k)=\det((q-F) \vert X(T_{\ol{k}}))$, where $F$ is the Frobenius in $\Gal(\ol{k}/k)$.  (For a proof of
this, see \cite[\S 1.5]{Oes}.)  Since $T$ splits over a finite extension, the action of $F$ on $X(T_{\ol{k}})$ has
finite order and so its eigenvalues are roots of unity.  We thus have $\# T(k)=\prod_{i=1}^r (q-\zeta_i)$ where each
$\zeta_i$ is a root of unity, from which the lemma easily follows.
\end{proof}

\begin{lemma}
\label{reg-3}
Let $G/k$ be a semi-simple group and let $T \subset G$ be a maximal torus defined over $k$.  Then $\# \{ \lambda \in
X(T_{\ol{k}}) \mid \Vert \lambda \Vert<n \}$ is bounded in terms of $\dim{G}$ and $n$.
\end{lemma}

\begin{proof}
The quantity $\# \{ \lambda \in X(T_{\ol{k}}) \mid \Vert \lambda \Vert <n \}$ depends only on $n$ and the root
datum associated to $(G_{\ol{k}}, T_{\ol{k}})$.  Since there are only finitely many semi-simple root data of a given
dimension, the result follows.
\end{proof}

\begin{lemma}
\label{reg-4}
Let $G/k$ be a semi-simple group of rank $r$, $T \subset G$ a maximal torus defined over $k$ and $\lambda \in
X(T_{\ol{k}})$ a non-zero character satisfying $\Vert \lambda \Vert<n$.  Then the kernel of the map
$\lambda:T(k) \to \ol{k}^{\times}$ has order at most $C (q+1)^{r-1}$ for some constant $C$ depending only on $n$ and
$\dim{G}$.
\end{lemma}

\begin{proof}
For a subset $S$ of $\{\lambda \in X(T_{\ol{k}}) \mid \Vert \lambda \Vert<n \}$ let $C(S)$ denote the
cardinality of the torsion of the quotient of $X(T_{\ol{k}})$ by the subgroup generated by $S$.  Let $C$ be the least
common multiple of the $C(S)$ over all $S$.  Since $C$ only depends upon $n$ and the root datum associated to
$(G_{\ol{k}}, T_{\ol{k}})$, it can be bounded in terms $n$ and $\dim{G}$.

Now, say the character $\lambda$ of $T_{\ol{k}}$ is defined over the extension $k'/k$.  Then $\lambda$ defines a map
$T \to \Res_{k'/k}(\G_m)$, where $\Res$ denotes restriction of scalars.  The kernel of $\lambda$ is a diagonalizable
group scheme whose character group is the cokernel of the map $f:\Z[\Gal(k'/k)] \to X(T_{\ol{k}})$ given by $\sigma
\mapsto \sigma \cdot \lambda$ (note that $\Z[\Gal(k'/k)]$ is the character group of $\Res_{k'/k}(\G_m)$).  The image of
$f$ is spanned by the $\Gal(\ol{k}/k)$ orbit of $\lambda$.  Since $\Vert \cdot \Vert$ is Galois invariant, it
follows that the torsion in the cokernel of $f$ has order at most $C$.  Furthermore, since $\lambda$ is non-zero, we see
that the rank of the cokernel of $f$ is at most $r-1$.

We have thus shown that the kernel of $T \to \Res_{k'/k}(\G_m)$ is an extension of a finite group scheme of order at
most $C$ by a torus of rank at most $r-1$.  It follows from Lemma~\ref{reg-2} that the set of $k$-points of the kernel
--- which is identified with the kernel of $\lambda:T(k) \to \ol{k}^{\times}$ ---
has cardinality at most $C (q+1)^{r-1}$, as was to be shown.
\end{proof}

We now prove the proposition.

\begin{proof}[Proof of Proposition~\ref{reg-1}]
Let $S$ be the set of all non-zero $\lambda \in X(T_{\ol{k}})$ such that $\Vert \lambda \Vert<2n$ and let $N$
be the cardinality of $S$.  We first claim that
\begin{displaymath}
T(k) \not \subset \bigcup_{\lambda \in S} \ker{\lambda}.
\end{displaymath}
Of course, this is equivalent to $T(k) \ne \bigcup_{\lambda \in S} (\ker{\lambda} \cap T(k))$.  To see this, we look at
the cardinality of each side.  The right side is a union of $N$ sets, each of which has cardinality at most
$C (q+1)^{r-1}$, while the left side has cardinality at least $(q-1)^r$ by Lemma~\ref{reg-2}.  Since $N$ and $C$ are
small compared to $q$ (by Lemma~\ref{reg-3} and Lemma~\ref{reg-4}), the claim follows.

Now, pick an element $g \in T(k)$ such that $g \not \in \bigcup_{\lambda \in S} \ker{\lambda}$.  Let $\lambda$ and
$\lambda'$ be distinct elements of $X(T_{\ol{k}})$ each of which has $\Vert \cdot \Vert<n$.  Then $\lambda
-\lambda'$ belongs to $S$.  Thus $(\lambda/\lambda')(g) \ne 1$ and so $\lambda(g) \ne \lambda'(g)$.  Therefore,
$\lambda \mapsto \lambda(g)$ is injective on those $\lambda$ with $\Vert \cdot \Vert<n$.
\end{proof}

\section{Bigness for algebraic representations}

In this section we prove that ``small'' algebraic representations have big image.  The main result is the following:

\begin{proposition}
\label{algbig-1}
Let $k$ be a finite field, let $G/k$ be a reductive group and let $\rho$ be an absolutely irreducible representation of
$G$ on a $k$-vector space $V$.  Assume $\ell=\chr{k}$ is large compared to $\dim{V}$ and $\Vert V \Vert$.  Then
$\rho(G(k))$ is a big subgroup of $\GL(V)$.
\end{proposition}

\begin{proof}
Let $G_1=G^{\der}$, a semi-simple group.  The group $G_1(k)$ is a normal subgroup of $G(k)$.  The quotient is a
subgroup of $(G/G_1)(k)$, which is prime to $\ell$ since $G/G_1$ is a torus.  Thus it suffices by
Proposition~\ref{elem-1} to show that $\rho(G_1(k))$ is big.

Let $H$ be the kernel of $\rho \vert_{G_1}$ and let $H'$ be the reduced subscheme of the identity component of $H$.
Then $H'$ is a closed normal subgroup of $G_1$ by Lemma~\ref{rep-bg-8} and is smooth, since it is reduced.  Put
$G_2=G_1/H'$.  The map $\rho$ factors through $G_2$.  By Lang's theorem, the natural map $G_1(k) \to G_2(k)$ is
surjective.  Thus $\rho(G_1(k))=\rho(G_2(k))$ and so it is enough to show that $\rho(G_2(k))$ is big.  Note that the
kernel of $\rho \vert_{G_2}$ is finite, and so the dimension of $G_2$ can be bounded in terms of the dimension of $V$.

Let $G_3$ be the universal cover of $G_2$.  The image of $G_3(k)$ in $G_2(k)$ is normal and the quotient is a
subgroup of $H^1(k, Z)$, where $Z=\ker(G_3 \to G_2)$.  Now, the order of $Z$ divides the order of the center of
$(G_3)_{\ol{k}}$, which can be computed in terms of the root datum of $(G_3)_{\ol{k}}$.  Since the dimension of
$G_3$ is bounded in terms of that of $V$ and there are only finitely many root data of a given dimension, it
follows that the order of $Z$ can be bounded in terms of the dimension of $V$.  Thus since $\ell$ is large
compared to the dimension of $V$, we find that the order of $Z$ is prime to $\ell$.  It follows that the index of
$G_3(k)$ in $G_2(k)$ is prime to $\ell$.  Thus it is enough to show that $\rho(G_3(k))$ is big.

We have thus shown that if $\rho(G_3(k))$ is big then so is $\rho(G(k))$.  Now, since $\rho$ is
an absolutely irreducible representation of $G$ the center of $G$ acts by a character under $\rho$.
Thus the restriction of $\rho$ to $G_1$ is still absolutely irreducible.  Therefore $\rho$ defines an absolutely
irreducible representation of $G_3$.  Furthermore, the norm of $V$ as a representation of $G$ is equal to the norm
of $V$ as a representation of $G_3$ by Proposition~\ref{rep-bg-0}.  We have thus shown that it suffices to prove the
proposition when $G$ is a simply connected, semi-simple group and $\ker{\rho}$ is a finite subgroup of $G$.  We now
begin the proof proper.

As $G$ is semi-simple and simply connected, it is a product of simple simply connected groups $G_i$.  Each $G_i$,
being simple and simply connected, is of the form $\Res_{k_i/k}(G_i')$ where $k_i$ is a finite extension of $k$ and
$G_i'$ is an absolutely simple simply connected group over $k_i$ (this is explained in \S 6.21(ii) of \cite{BoT}).  We
have $G_i(k)=G_i'(k_i)$.  Let $Z_i'$ be the center of $G_i'$.  By \cite[Thm.~11.1.2]{Carter} and
\cite[Thm.~14.4.1]{Carter} the group $G_i'(k_i)/Z_i'(k_i)$ is
simple and non-abelian.  As we have previously explained, the order of $Z_i'$ can be bounded by $\dim{G_i'}$, which is
in turn bounded by $\dim{V}$.  Thus, by our assumptions, the order of $Z_i'$ is small compared to $\ell$.  We therefore
find that the Jordan-Holder constituents of $G(k)$ are all simple groups of Lie type and abelian groups of the form
$\Z/p \Z$ with $p$ a prime that is small compared to $\ell$.  In particular, $\Z/\ell \Z$ is not a Jordan-Holder
constituent of $G(k)$ and therefore not a constituent of the quotient $\rho(G(k))$.  Thus $\rho(G(k))$ does not have a
quotient of $\ell$-power order, as such a quotient would be solvable and have a quotient isomorphic to $\Z/\ell \Z$.
This shows that $\rho(G(k))$ satisfies (B1).

Proposition~\ref{rep-bg-3} shows that $V$ is absolutely irreducible as a representation of $G(k)$ and so $\rho(G(k))$
satisfies (B2).

We now examine $H^1(\rho(G(k)), \ad^{\circ}{V})$.  By Proposition~\ref{rep-bg-1} and Proposition~\ref{rep-bg-3}, we have
$H^1(G(k), \ad{V})=0$ since this group classifies self-extensions of $V$ and any such extension is semi-simple.  Since
$\ell$ is large compared to $\dim{V}$ this implies $H^1(G(k), \ad^{\circ}{V})=0$, as $\ad^{\circ}(V)$ is a
summand of $\ad{V}$.  Let $H$ be the kernel of $\rho$.  We have an exact sequence
\begin{displaymath}
1 \to H(k) \to G(k) \to \rho(G(k)) \to 1
\end{displaymath}
and thus we get an injection
\begin{displaymath}
H^1(\rho(G(k)), (\ad^{\circ}{V})^{H(k)}) \to H^1(G(k), \ad^{\circ}{V}).
\end{displaymath}
The group on the right vanishes and so the group on the left does too.  Since $H(k)$ acts trivially on $V$, it acts
trivially on $\ad^{\circ}{V}$.  Thus $H^1(\rho(G(k)), \ad^{\circ}{V})=0$ and so $\rho(G(k))$ satisfies (B3).

We now turn to condition (B4).  As every reductive group over a finite field is quasi-split (see
\cite[Prop.~16.6]{Borel}), we can pick a Borel subgroup $B$ of $G$ defined over $k$.  Let $T$ be a maximal torus of
$B$, which is automatically a maximal torus of $G$, and let $U$ be the
unipotent radical of $B$.  By Proposition~\ref{reg-1} we can pick an element $g$ of $T(k)$ such that $\lambda(g) \ne
\lambda'(g)$ for any two distinct characters $\lambda$ and $\lambda'$ of $T_{\ol{k}}$ which are weights of $V_{\ol{k}}$.
Now, the space $V_0=V^U$ is one dimensional and stable under the action of $T$.  Let $\lambda_0:T \to \G_m$ give the
action of $T$ on $V_0$.  Then $\lambda_0$ is the highest weight of $V_{\ol{k}}$, and thus occurs as a weight in this
representation with multiplicity one.  Put $\alpha=\lambda_0(g)$, an element of $k^{\times}$.  We then have
$V_{g, \alpha}=V_0$, and so the $\alpha$-generalized eigenspace of $g$ is one dimensional.  To show that the image of
$G(k)$ is big, it thus suffices to show that any irreducible $G(k)$-submodule of $\ad{V}$ has non-zero projection to
$\ad{V_0}$.

Thus let $W$ be an irreducible $G(k)$-submodule of $\ad{V}$.  To show that the image of $W$ in $\ad{V_0}$
is non-zero it suffices to show that the image of $\ol{W}$ in $\ad{\ol{V}_0}$ is, where the bar denotes
$-\otimes \ol{k}$.  Let $\ol{U}$ be an irreducible $G(k)$-submodule of $\ol{W}$.  It is enough, of course, to show that
the image of $\ol{U}$ in $\ad{\ol{V}_0}$ is non-zero.  Now, $\ol{U}$ is an irreducible $G(k)$-submodule
of $\ad{\ol{V}}$, and so, by Proposition~\ref{rep-bg-3}, it is an irreducible $G$-submodule of $\ad{\ol{V}}$.  Thus to
show that the image of $G(k)$ in $\GL(V)$ satisfies (B4) it is enough to prove the following:  \emph{every irreducible
$G$-submodule of $\ad{\ol{V}}$ has non-zero image in $\ad{\ol{V}_0}$.}  This is established in the
following lemma.
\end{proof}

\begin{lemma}
Let $k$ be the algebraic closure of a finite field, let $G/k$ be a semi-simple group and let $(\rho, V)$ be an
irreducible representation of $G$ with $\dim{V}$ and $\Vert V \Vert$ small compared to $\chr{k}$.  Let $V_0$
be the highest weight space of $V$.  Then every irreducible submodule of $\ad{V}$ has non-zero projection to
$\ad{V_0}$.
\end{lemma}

\begin{proof}
By the same reductions used in the proof of the proposition we can assume that $\ker{\rho}$ is finite, and so
$\dim{G}$ is bounded in terms of $\dim{V}$.  Pick a maximal torus of $G$
and a system of positive roots.  For a weight $\lambda$ let $V_{\lambda}$ denote the $\lambda$-weight space of $V$.  Let
$\lambda_0$ be the highest weight for $V$ and let $V_0$ be the $\lambda_0$-weight space.  For a root $\alpha$ we pick an
element $X_{\alpha}$ of $\Lie(G)$ which spans the $\alpha$ root space.  Any positive element $\lambda$ of the root
lattice has a unique expression $\lambda=\sum n_i \alpha_i$ where the $n_i$ are non-negative integers and the $\alpha_i$
are the simple roots.  We let $\len{\lambda}$ be the sum of the $n_i$.

By a \emph{simple tuple} we mean an ordered tuple $\ul{\alpha}=(\alpha_1, \ldots, \alpha_n)$ consisting of simple roots.
For such a tuple $\ul{\alpha}$ we put $\vert \ul{\alpha} \vert=\sum \alpha_i$.  Note that $\len{\vert \ul{\alpha} \vert}
=n$.  We let $X_{\ul{\alpha}}$ (resp.  $Y_{\ul{\alpha}}$) denote the product $X_{\alpha_1} \cdots X_{\alpha_n}$ (resp.
$X_{-\alpha_1} \cdots X_{-\alpha_n}$), regarded as an element of the universal enveloping algebra.

Given a weight $\lambda$ for which $V_{\lambda}$ is non-zero the difference $\lambda_0-\lambda$ is positive and lies in
the root lattice.  For a simple tuple $\ul{\alpha}$ with $\vert \ul{\alpha} \vert=\lambda_0-\lambda$ the operator
$X_{\ul{\alpha}}$ maps $V_{\lambda}$ into $V_0$.  The resulting map
\begin{displaymath}
V_{\lambda} \to \bigoplus_{\vert \ul{\alpha} \vert=\lambda_0-\lambda} V_0
\end{displaymath}
is injective.  (Proof:  By Proposition~\ref{rep-bg-2}, the only vector annihilated by all of the $X_{\alpha}$ is the
highest weight vector.  Thus if $\lambda \ne \lambda_0$ and $v$ belongs to $V_{\lambda}$ then we can find some $\alpha$
such that $X_{\alpha} v$ is non-zero.  We can thus move $v$ closer to the $\lambda_0$-weight space.  By induction on
$\len(\lambda_0-\lambda)$ we can therefore find $\alpha_1, \ldots, \alpha_n$ such that $X_{\alpha_n} \cdots X_{\alpha_1}
v$ is non-zero and belongs to $V_0$.)  We can thus pick $m=\dim{V_{\lambda}}$ simple tuples $\ul{\alpha}_1, \ldots,
\ul{\alpha}_m$ for which the resulting map is injective.  We can then pick a basis $\{v_i\}$ of $V_{\lambda}$ such that
$v_i$ belongs to the kernel of $X_{\ul{\alpha}_j}$ whenever $i \ne j$ but does not belong to the kernel of
$X_{\ul{\alpha}_i}$. We call such a basis \emph{admissible}.  Note that in $V^*$ the space $V_0^*$ is a lowest weight
space.  The same process as above, but with $X_{\ul{\alpha}}$ replaced by $Y_{\ul{\alpha}}$, yields the notion of an
admissible basis for $V_{\alpha}^*$.

Let $W$ be an irreducible submodule of $\ad{V}$.  Let $p:W \to V_0 \otimes V^*$ be the natural projection.  We first
show that $p(W)$ is non-zero. Among those weights $\lambda$ for which the projection $W \to V_{\lambda} \otimes V^*$ is
non-zero, pick one for which $\len(\lambda_0-\lambda)$ is minimal.  Let $w$ be an element of $W$ which has non-zero
projection to $V_{\lambda} \otimes V^*$ and write
\begin{displaymath}
w=\left( \sum v_i \otimes v_i^* \right)+ v'
\end{displaymath}
where $\{v_i\}$ is an admissible basis of $V_{\lambda}$, the $v_i^*$ belong to $V^*$ and $v'$ belongs to the complement
of $V_{\lambda} \otimes V^*$.  Let $\ul{\alpha}_i$ be the simple tuples yielding the basis $v_i$.  Let 1 denote an index
such that $v_1^*$ is non-zero.  We then have
\begin{equation}
\label{eq1}
p(X_{\ul{\alpha}_1} w)=(X_{\ul{\alpha}_1} v_1) \otimes v_1^*
\end{equation}
(explained below).  Since the right side is non-zero, it follows that $p(W)$ is non-zero.

We now explain why \eqref{eq1} holds.  Recall that if $X$ is an element of $\Lie(G)$ then the formula for how $X$ acts
on a pure tensor is
\begin{displaymath}
X(v \otimes w)=(Xv) \otimes w + v \otimes (Xw).
\end{displaymath}
Thus when we apply $X_{\ul{\alpha}_1}$ to a pure tensor $v \otimes w$ we get a sum of terms and in each term some
$X_{\alpha_{1, i}}$ land on $v$ and some land on $w$.  We now examine $X_{\ul{\alpha}_1} v$.  First consider the $v'$
part.  Write $v'=\sum v'_i \otimes u'_i$ where $v'_i$ has weight $\mu_i$.  If $\ul{\alpha}'$ is any sub-sequence of
$\ul{\alpha}_1$ then $X_{\ul{\alpha}'} v_i'$ lands in the $\mu_i+\vert \ul{\alpha}' \vert$ weight space.  If
$\ul{\alpha}'$ is not all of $\ul{\alpha}_1$ then this cannot equal $\lambda_0$ for length reasons.  Even if
$\ul{\alpha}'$ is all of $\ul{\alpha}_1$ this is not equal to $\lambda_0$ since $\lambda_0=\vert \ul{\alpha}_1 \vert
+\lambda$ and no $\mu_i$ is equal to $\lambda$.  Thus $p(X_{\ul{\alpha}_1} v')=0$.  We now consider the first term in
$v$.  The same length argument shows that the only way to land in $V_0 \otimes V^*$ is to have all of
$X_{\ul{\alpha}_1}$ land on the first factor.  However, $X_{\ul{\alpha}_1}$ kills $v_i$ for $i \ne 1$.  We have thus
proved \eqref{eq1}.

We now show that the image of the projection $q:W \to V_0 \otimes V_0^*$ is non-zero.  Among those weights $\lambda$
for which the projection $W \to V_0 \otimes V_{\lambda}^*$ is non-zero, pick one for which $\len(\lambda_0-\lambda)$ is
minimal.  (Such a weight exists by the previous paragraphs.)  Let $w$ be an element of $W$ which has non-zero projection
to $V_0 \otimes V_{\lambda}^*$.  We may as well assume that $w$ has weight $\lambda_0-\lambda$.  We can thus write
\begin{displaymath}
w=\left( \sum v_i \otimes v_i^* \right) +v'
\end{displaymath}
where the $v_i$ belong to $V_0$, $\{v_i^*\}$ is an admissible basis of $V_{\lambda}^*$ and $v'$ belongs to the
complement of $V_0 \otimes V^*$.  Let $\ul{\alpha}_i$ be the simple tuples yielding the basis $v_i^*$.  Let 1 denote
an index such that $v_1$ is non-zero.  We then have
\begin{equation}
\label{eq2}
q(Y_{\ul{\alpha}_1} w)=v_1 \otimes (Y_{\ul{\alpha}_1} v_1^*)
\end{equation}
(explained below).  Since the right side is non-zero, it follows that $q(W)$ is non-zero, proving the proposition.

We now explain \eqref{eq2}.  The point is that, since $Y_{\ul{\alpha}_1}$ is a lowering operator, the only way for a
term of $Y_{\ul{\alpha}_1} w$ to have its first factor in $V_0$ is if $Y_{\ul{\alpha}_1}$ lands entirely on the second
factor.  Of course, none of the terms in $v'$ have their first factor in $V_0$ to begin with, so they will not after
applying $Y_{\ul{\alpha}_1}$.  As for the first term, $Y_{\ul{\alpha}_1}$ kills $v_i^*$ for $i \ne 1$.  This proves
\eqref{eq2}.
\end{proof}

\section{Bigness for nearly hyperspecial groups}

Throughout this section $K$ denotes a finite extension of $\Q_{\ell}$, $\ms{O}_K$ its ring of integers and $k$ its
residue field.

We begin by recalling some definitions.  Let $G/K$ be a reductive group.  The group $G$ is \emph{quasi-split} if it has
a Borel subgroup.  It is \emph{unramified} if it is quasi-split and it splits over an unramified extension of $K$.  A
subgroup $\Gamma \subset G(K)$ is \emph{hyperspecial} if there exists a reductive group $\wt{G}/\ms{O}_K$ with generic
fiber $G$ such that $\Gamma=\wt{G} (\ms{O}_K)$.  Hyperspecial subgroups of $G(K)$ are maximal compact subgroups.  The
group $G(K)$ has a hyperspecial subgroup if and only if $G$ is unramified.  Let $G^{\ad}$ be the adjoint group of $G$
and let $G^{\scn}$ be the simply connected cover of $G^{\ad}$.  We have maps
\begin{displaymath}
\xymatrix{
G \ar[r]^-{\sigma} & G^{\ad} & G^{\scn} \ar[l]_>>>>>>{\tau}. }
\end{displaymath}
We say that a subgroup $\Gamma \subset G(K)$ is \emph{nearly hyperspecial} if $\tau^{-1}(\sigma(\Gamma))$ is a
hyperspecial subgroup of $G^{\scn}(K)$.  (This is not a standard term.)

The purpose of this section is to prove the following proposition:

\begin{proposition}
\label{nhyper-5}
Let $\rho:\Gamma \to \GL_n(K)$ be a continuous representation of the profinite group $\Gamma$.  Assume:
\begin{itemize}
\item The characteristic $\ell$ of $k$ is large compared to $n$.
\item The restriction of $\rho$ to any open subgroup of $\Gamma$ is absolutely irreducible.
\item The index of $G^{\circ}$ in $G$ is small compared to $\ell$, where $G$ is the Zariski closure of $\rho(\Gamma)$.
\item The subgroup $\rho(\Gamma) \cap G^{\circ}(K)$ of $G^{\circ}(K)$ is nearly hyperspecial.
\end{itemize}
Then $\ol{\rho}(\Gamma)$ is a big subgroup of $\GL_n(k)$.
\end{proposition}

We remark that the second condition in the proposition, that the restriction of $\rho$ to any open subgroup remain
absolutely irreducible, is equivalent to the condition that the representation of $G^{\circ}$ on $V$ be absolutely
irreducible.  We need some auxiliary lemmas to prove the proposition.  We begin with the following one.

\begin{lemma}
\label{nhyper-2}
Let $\wt{G}/\ms{O}_K$ be a simply connected semi-simple group and let $\sigma$ be an automorphism of the generic fiber
$G=\wt{G}_K$ such that $\sigma$ maps $\wt{G}(\ms{O}_K)$ into itself.  Then for any tamely ramified finite extension
$L/K$ the automorphism $\sigma$ maps $\wt{G}(\ms{O}_L)$ into itself.
\end{lemma}

\begin{proof}
The group $\wt{G}(\ms{O}_K)$ fixes a point $x$ on the building $B(G, K)$ by \cite[\S 2.3.1]{Tits} or
\cite[\S 4.6.31]{BrT} which is known to be unique.  Similarly, the group $\wt{G}(\ms{O}_L)$ fixes a unique point $x'$ on
the building $B(G, L)$.  Furthermore $\wt{G}(\ms{O}_K)$ (resp.\ $\wt{G}(\ms{O}_L)$) is the full stabilizer of $x$
(resp.\ $x'$) since $\wt{G}(\ms{O}_K)$ (resp.\ $\wt{G}(\ms{O}_L)$) is maximal compact (\cite[\S 3.2]{Tits}).  We now
claim that under the natural inclusion $B(G, K) \to B(G, L)$ the point $x$ is identified with $x'$.  To see this, first
note that if $\tau$ is an element of $\Gal(L/K)$ then $\wt{G}(\ms{O}_L)$ fixes $\tau x'$ and so $\tau x'=x'$ by the
uniqueness of $x'$.  Thus $x'$ is fixed by $\Gal(L/K)$ and therefore belongs to $B(G, K)$ by \cite[\S 2.6.1]{Tits} (this
uses the hypothesis that $L/K$ is tamely ramified).  Since $x'$ is fixed by $\wt{G}(\ms{O}_L)$ it is certainly also
fixed by the subgroup $\wt{G}(\ms{O}_K)$.  By the uniqueness of $x$ we conclude $x=x'$.

Now, the automorphism $\sigma$ of $G$ acts on $B(G, K)$ and $B(G, L)$ and respects the inclusion map.  As $\sigma$
carries $\wt{G}(\ms{O}_K)$ into itself it must fix $x$.  It therefore also fixes $x'$ and so must carry its stabilizer,
$\wt{G}(\ms{O}_L)$, into itself.  This proves the lemma.
\end{proof}

We can now prove the following:

\begin{lemma}
\label{nhyper-1}
Let $\Gamma$ be a profinite group and let $\rho$ be an absolutely irreducible representation of $\Gamma$ on a
$K$-vector space $V$.  Assume that the Zariski closure $G$ of $\rho(\Gamma)$ is connected and that $\rho(\Gamma)$ is a
nearly hyperspecial subgroup of $G(K)$.  Then we can find:
\begin{itemize}
\item a $\Gamma$-stable lattice $\Lambda$ in $V$;
\item a semi-simple group $\wt{G}/\ms{O}_K$ with generic fiber equal to $G^{\scn}$; and
\item a representation $r:\wt{G} \to \GL(\Lambda)$ which induces the natural map $G^{\scn} \to G$ on the generic fiber,
\end{itemize}
such that $\ms{O}_K^{\times} \cdot r(\wt{G}(\ms{O}_K))$ is an open normal subgroup of $\ms{O}_K^{\times} \cdot
\rho(\Gamma)$, the index of which can be bounded in terms of $\dim{V}$.  Necessarily, the generic fiber of $r$ is an
absolutely irreducible representation of $\wt{G}_K$ on $V$.
\end{lemma}

\begin{proof}
The group $G$ is a reductive (and in particular connected) group, by hypothesis.  Since $\rho$ is absolutely
irreducible, the center $Z$ of $G$ is contained in the center of $\GL(V)$.  We have maps
\begin{displaymath}
\xymatrix{
G \ar[r]^-{\sigma} & G^{\ad} & G^{\scn} \ar[l]_>>>>>>{\tau} \ar[ld] \\
& G^{\der}. \ar[ul] \ar[u] }
\end{displaymath}
By hypothesis, $\tau^{-1} (\sigma(\rho(\Gamma)))$ is a hyperspecial subgroup of $G^{\scn}$.  Thus we can find a
semi-simple group $\wt{G}/\ms{O}_K$ with generic fiber $G^{\scn}$ such that $\wt{G}(\ms{O}_K)=\tau^{-1}(\sigma(\rho
(\Gamma)))$.

Let $r:G^{\scn} \to G$ be the natural map; this factors through $G^{\der}$ in the above diagram.  Let $U$ be the image
of $G^{\scn}(K)$ under $\tau$.  It is an open normal subgroup of $G^{\ad}(K)$, the index of which can be bounded in
terms of $\dim{G}$ and thus $\dim{V}$ (by arguments similar to those used in the third paragraph of the proof of
Proposition~\ref{algbig-1}).  Now, we have
\begin{displaymath}
\sigma(r(\wt{G}(\ms{O}_K)))=\tau(\wt{G}(\ms{O}_K))=\sigma(\rho(\Gamma)) \cap U.
\end{displaymath}
Applying $\sigma^{-1}$, we find
\begin{displaymath}
K^{\times} \cdot r(\wt{G}(\ms{O}_K)) = K^{\times} \cdot (\rho(\Gamma) \cap \sigma^{-1}(U)).
\end{displaymath}
Since $r(\wt{G}(\ms{O}_K))$ and $\rho(\Gamma) \cap \sigma^{-1}(U)$ are both compact, it follows that
\begin{displaymath}
\ms{O}_K^{\times} \cdot r(\wt{G}(\ms{O}_K)) = \ms{O}_K^{\times} \cdot (\rho(\Gamma) \cap \sigma^{-1}(U)).
\end{displaymath}
Thus $\ms{O}_K^{\times} \cdot r(\wt{G}(\ms{O}_K))$ is an open normal subgroup of $\ms{O}_K^{\times} \cdot \rho(\Gamma)$,
the index of which can be bounded in terms of $\dim{V}$.

We now claim that for any finite unramified extension $L/K$ the group $\rho(\Gamma)$ normalizes $\ms{O}_L^{\times} \cdot
r(\wt{G}(\ms{O}_L))$.  To see this, let $\gamma$ be an element of $\rho(\Gamma)$.  Write $\ol{\gamma}$ for the image of
$\gamma$ in $G^{\ad}(K)$ under $\sigma$.  Thus $\ol{\gamma}$ gives an automorphism of $G^{\scn}$, which we denote by $x
\mapsto \ol{\gamma} x \ol{\gamma}^{-1}$.  Now, let $x$ be an element of $G^{\scn}(L)$.  We then have
\begin{displaymath}
\gamma r(x) \gamma^{-1}=z \cdot r(\ol{\gamma} x \ol{\gamma}^{-1}),
\end{displaymath}
for some $z \in \ms{O}_L^{\times}$, as is easily seen by applying $\sigma$.  It thus suffices to show that conjugation
by $\ol{\gamma}$ carries $\wt{G}(\ms{O}_L)$ into itself.  By Lemma~\ref{nhyper-2} it suffices to show that $\ol{\gamma}$
carries $\wt{G}(\ms{O}_K)$ into itself.  Thus let $x$ be an element of $\wt{G}(\ms{O}_K)$.  Using the above formula and
the fact that $\gamma$ normalizes $\ms{O}_K^ {\times}\cdot r(\wt{G}(\ms{O}_K))$, we can find an element $y$ of
$\wt{G}(\ms{O}_K)$ and an element $z$ of $\ms{O}_K^{\times}$ such that $r(\ol{\gamma} x \ol{\gamma}^{-1})=z r(y)$.  It
thus follows that $\ol{\gamma} x \ol{\gamma}^{-1}=z' y$ for some element $z'$ of the $K$-points of center of $G^{\scn}$.
However, $z'$ must be contained in $\wt{G}(\ms{O}_K)$ since it belongs to a compact central group and $\wt{G}(\ms{O}_K)$
is maximal compact.  Thus $\ol{\gamma} x \ol{\gamma}^{-1}$ belongs to $\wt{G}(\ms{O}_K)$.

Now, the group $\wt{G}(\ms{O}_K^{\un})$ is bounded in the sense of \cite[\S 2.2.1]{Tits}.  Thus, arguing as in
\cite[\S 1.12]{Lar}, we can find a lattice $\Lambda' \subset V$ such that $\Lambda' \otimes \ms{O}_K^{\un}$ is stable
under the action of $\wt{G}(\ms{O}_K^{\un})$.  Now, we have shown that $\ms{O}_K^{\times} \cdot r(\wt{G}(\ms{O}_K))$ has
finite index in $\ms{O}_K^{\times} \cdot \rho(\Gamma)$.  Let $\gamma_1, \ldots, \gamma_n$ be coset representatives and
put
\begin{displaymath}
\Lambda=\sum_{i=1}^n \gamma_i \cdot \Lambda'.
\end{displaymath}
Thus $\Lambda$ is a lattice in $V$.  It is easy to see that $\Gamma$ maps $\Lambda$ into itself and
$\wt{G}(\ms{O}_K^{\un})$ maps $\Lambda \otimes \ms{O}_K^{\un}$ into itself.  Following the argument in
\cite[\S 1.12]{Lar} once again, we see that $r:\wt{G}_K \to \GL(V)$ lifts to a map $\wt{G} \to \GL(\Lambda)$, which
we still call $r$.  This completes the proof of the proposition.
\end{proof}

We can now prove the proposition.

\begin{proof}[Proof of Proposition~\ref{nhyper-5}]
Let $\Gamma$, $\rho$ and $G$ be as in the statement of the proposition, and let $V=K^n$ be the representation space of
$\rho$.  We must show that $\ol{\rho}(\Gamma)$ is
big.  Let $\Gamma^{\circ}=\Gamma \cap G^{\circ}(K)$.  Then $\Gamma^{\circ}$ is a normal subgroup of $\Gamma$ of prime to
$\ell$ index (since the number of components of $G$ is assumed small compared to $\ell$).  It is therefore enough, by
Proposition~\ref{elem-1}, to show that $\ol{\rho}(\Gamma^{\circ})$ is big.  Replacing $\Gamma$ by $\Gamma^{\circ}$, it
thus suffices to prove the proposition under the assumption that the Zariski closure $G$ of $\rho(\Gamma)$ is connected.

Let $\wt{G}/\ms{O}_K$, $\Lambda$ and $r$ be as in Lemma~\ref{nhyper-1}.  Let $\ol{\rho}$ be the representation of
$\Gamma$ on $U=\Lambda \otimes_{\ms{O}_K} k$.  By Proposition~\ref{rep-bg-4}, the representation of $\wt{G}_k$ on $U$
is absolutely irreducible and its norm is bounded in terms of $\dim{V}$.  It thus follows from
Proposition~\ref{algbig-1} that $r(\wt{G}(k))$ is a big subgroup of $\GL(U)$.  Now, $\ms{O}_K^{\times} \rho(\Gamma)$
contains $\ms{O}_K^{\times} r(\wt{G}(\ms{O}_K))$ as a normal subgroup of index prime to $\ell$.  Taking the image of
each group in $\GL(U)$, we find that $k^{\times} \ol{\rho}(\Gamma)$ contains $k^{\times} r(\wt{G}(k))$ as a normal
subgroup of index prime to $\ell$.  (Note that the image of of $r(\wt{G}(\ms{O}_K))$ in $\GL(U)$ is equal to
$r(\wt{G}(k))$ since $\wt{G}$ is smooth over $\ms{O}_K$.)  Since $r(\wt{G}(k))$ is big, we conclude the same for
$\ol{\rho}(\Gamma)$ by Proposition~\ref{elem-1} and Proposition~\ref{elem-2}.
\end{proof}

\section{Groups with Frobenii and compatible systems}
\label{s:compsys}

A \emph{group with Frobenii} is a pair $(\Gamma, \ms{F})$ consisting of a profinite group $\Gamma$ and a dense set of
elements $\ms{F}$ of $\Gamma$ indexed by a set $P$.  The elements of $\ms{F}$ are called ``Frobenius elements.''
The motivating example of a group with Frobenii is the Galois group of a global field.  Let $F$ be a global field (that
is, a finite extension of $\F_p(t)$ or of $\Q$) and let $\Gamma$ be its absolute Galois group.  For each place $v$ of
$F$ choose a Frobenius element $\Frob_v$ and let $\ms{F}$ be the set of all the $\Frob_v$.  Then $(\Gamma, \ms{F})$
is a group with Frobenii.

Let $\Gamma$ be a group with Frobenii, let $E$ be a number field, let $L$ be a set of places of $E$ and for each
$w \in L$ let $\rho_w:\Gamma \to \GL_n(E_w)$ be a continuous representation.  We say that the $\rho_w$ form a
\emph{compatible system} (with coefficients in $E$) if for each Frobenius element $F \in \ms{F}$ there exists a finite
set of places $L_F \subset L$ (the ``bad places'' for $F$) such that the following conditions hold:
\begin{itemize}
\item The characteristic polynomial of $F$ has coefficients in $E$ and is independent of $w$ for good $w$.  Precisely,
given $F \in \ms{F}$ there is a polynomial $p$ with coefficients in $E$ such that for all places $w \in L \setminus
L_F$ the characteristic polynomial of $\rho_w(F)$ is equal to $p$.
\item For any finite subset $L'$ of $L$ the  Frobenii for which all primes in $L'$ are good form a dense set in
$\Gamma$.  That is, for any such $L'$ the set set $\{F \in \ms{F} \mid L' \cap L_F = \emptyset \}$ is dense.
\end{itemize}
By a ``compatible system of semi-simple representations'' we simply mean a compatible system in which each $\rho_w$
is semi-simple.  We call a set $L$ of places of $E$ \emph{full} if there exists a set of rational primes $P$ of
Dirichlet density one such that for all $\ell \in P$ all places of $E$ over $\ell$ belong to $L$.

\begin{proposition}
\label{compsys-1}
Let $\Gamma$ be a group with Frobenii and let $\{\rho_w\}_{w \in L}$ be a compatible system of $n$ dimensional
semi-simple representations of $\Gamma$ with coefficients in $E$, with $L$ a full set of places.  We assume that $E$
is Galois over $\Q$.  Let $G_w$ be the
Zariski closure of $\rho_w(\Gamma)$ in $\GL_n(E_w)$ and let $G_w^{\circ}$ be its identity component.  Then there is a
finite index subgroup $\Gamma^{\circ}$ of $\Gamma$ and a set of primes $P$ of $\Q$ of Dirichlet density $1/[E:\Q]$, all
of which split completely in $E$, such that if $w \in L$ lies over a prime in $P$ then:
\begin{enumerate}
\item The Zariski closure of $\rho_w(\Gamma^{\circ})$ is $G_w^{\circ}$.
\item The group $\rho_w(\Gamma^{\circ})$ is a nearly hyperspecial subgroup of $G_w^{\circ}(E_w)$.
\end{enumerate}
\end{proposition}

\begin{proof}\ \hskip -8pt{}\footnote{An argument similar to the one given here appeared earlier in \cite{BLGGT}.}
When $E=\Q$, the first statement is due to Serre (see \cite[Prop.~6.14]{LP}) and the second to Larsen (see \cite{Lar}).
We will deduce the statement for arbitrary $E$ from the $E=\Q$ case.
Let $P_0$ be the set of rational primes $\ell$ such that all places of $E$ above $\ell$ belong to $L$.  Then $P_0$ has
Dirichlet density one since $L$ is full.  For $\ell \in P$ define $\sigma_{\ell}=\bigoplus_{w \mid \ell} \rho_w$, where
here $\rho_w$ is regarded as a $\Q_{\ell}$ representation of dimension $n [E_w:\Q_{\ell}]$.  Then $\sigma_{\ell}$ is a
$\Q_{\ell}$ representation of $\Gamma$ of dimension $nm$, where $m=[E:\Q]$.  One easily sees that
$\{\sigma_{\ell}\}_{\ell \in P}$ forms a compatible system.

Let $H_{\ell}$ be the Zariski closure of the image of $\sigma_{\ell}$.  Applying the $E=\Q$ case of the proposition,
we can find a set of primes $P_1 \subset P_0$ of Dirichlet density one and a finite index subgroup $\Gamma^{\circ}$ of
$\Gamma$ such that for all $\ell \in P$ we have:  (1) the Zariski closure of $\sigma_{\ell}(\Gamma^{\circ})$ is
$H_{\ell}^{\circ}$; and (2) $\sigma_{\ell}(\Gamma^{\circ})$ is a nearly hyperspecial subgroup of $H_{\ell}^{\circ}
(\Q_{\ell})$.  Let $P$ be the set of primes in $P_1$ which split completely in $E$.  Let $\ell \in P$ and pick
$w \mid \ell$.  Since $E_w=\Q_{\ell}$, the representation $\rho_w$ is an $n$-dimensional $\Q_{\ell}$ representation,
and as such a summand of $\sigma_{\ell}$.  Thus
$\sigma_{\ell}(\Gamma^{\circ})$ surjects onto $\rho_w(\Gamma^{\circ})$, and so $H_{\ell}^{\circ}$ surjects onto the
Zariski closure of $\rho_w(\Gamma^{\circ})$.  It follows that the Zariski closure of $\rho_w(\Gamma^{\circ})$ is
connected.  Since $\rho_w(\Gamma^{\circ})$ has finite index in $\rho_w(\Gamma)$, the Zariski closure of the former
must be the connected component of the Zariski closure of the latter, namely $G_w^{\circ}$.  The following lemma
shows that $\rho_w(\Gamma^{\circ})$ is nearly hyperspecial in $G_w^{\circ}(\Q_{\ell})$.
\end{proof}

\begin{lemma}
Let $K/\Q_{\ell}$ be a finite extension, let $f:G \to H$ be a surjection of reductive groups over $K$ and let
$\Gamma$ be a nearly hyperspecial subgroup of $G(K)$.  Then $f(\Gamma)$ is a nearly hyperspecial subgroup of $H(K)$.
\end{lemma}

\begin{proof}
Consider the diagram
\begin{displaymath}
\xymatrix{
G^{\scn} \ar[r]^{\tau} \ar[d]^{f''} & G^{\ad} \ar[d]^{f'} & G \ar[l]_{\sigma} \ar[d]^f \\
H^{\scn} \ar[r]^{\tau'} & H^{\ad} & H \ar[l]_{\sigma'} }
\end{displaymath}
where $f''$ is the lift of $f'$.  Let $\Gamma'=f(\Gamma)$, $\Delta=\tau^{-1}(\sigma(\Gamma))$ and $\Delta'=
(\tau')^{-1}(\sigma'(\Gamma'))$.  We are given that $\Delta$ is hyperspecial and we want to show that $\Delta'$ is
hyperspecial.  One easily sees that $f''(\Delta) \subset \Delta'$ and that $\Delta'$ is compact.  Now, since $G^{\scn}$
and $H^{\scn}$ are simply connected semi-simple groups the map $f''$ is a projection onto a direct factor.  It follows
that $G^{\scn}=H^{\scn} \times H'$ for some group $H'$.  The following lemma shows that $\Delta=\Delta_1 \times
\Delta_2$ where $\Delta_1$ is a hyperspecial subgroup of $H^{\scn}(K)$ and $\Delta_2$ is a hyperspecial subgroup of
$H'(K)$.  We thus find $f''(\Delta) = \Delta_1 \subset \Delta'$.  Since $\Delta'$ is compact and $\Delta_1$ is maximal
compact, we have $\Delta'=\Delta_1$ and so $\Delta'$ is hyperspecial.
\end{proof}

\begin{lemma}
Let $K/\Q_{\ell}$ be a finite extension, let $H_1$ and $H_2$ be reductive groups over $K$ and let $\Delta$ be a
hyperspecial subgroup of $H_1(K) \times H_2(K)$.  Then $\Delta=\Delta_1 \times \Delta_2$ where $\Delta_i$ is a
hyperspecial subgroup of $H_i(K)$.
\end{lemma}

\begin{proof}
We thank Brian Conrad for this argument.
Let $\Delta=\wt{G}(\ms{O}_K)$ where $\wt{G}/\ms{O}_K$ is a reductive group with generic fiber $G$.  We wish to find
$\wt{G}_i$ such that $\Delta_i=\wt{G}_i(\ms{O}_K)$.  If $\wt{G}_i$ exists then it is necessarily the Zariski closure
of $G_i$ in $\wt{G}$ and thus unique.  To establish the existence of $\wt{G}_i$ we may therefore (by descent theory)
work \'etale locally
on $\ms{O}_K$.  We may therefore replace $\ms{O}_K$ by a cover and assume that $\wt{G}$ is split.  Let $\wt{T}$ be
a split maximal torus of $\wt{G}$.  Then the root datum for $(\wt{G}, \wt{T})$ is canonically identified with that
for $(G, T)$, where $T$ is the generic fiber of $\wt{T}$.  As the latter is a product, so is the former.  Thus
$\wt{G}=\wt{G}_1 \times \wt{G}_2$ where the generic fiber of $\wt{G}_i$ is $G_i$.  This establishes the lemma.
\end{proof}

\section{Bigness for compatible systems}
\label{s:bigsys}

We can now prove our main theorem:

\begin{theorem}
\label{csys-1}
Let $\Gamma$ be a group with Frobenii, let $E$ be a Galois extension of $\Q$, let $L$ be a full set of places of $E$
and for each
$w \in L$ let $\rho_w:\Gamma \to \GL_n(E_w)$ be a continuous representation.  Assume that $\{\rho_w\}_{w \in L}$ forms
a compatible system and that each $\rho_w$ is absolutely irreducible when restricted to any open subgroup of $\Gamma$.
Then there is a set of primes $P$ of $\Q$ of Dirichlet density $1/[E:\Q]$, all of which split completely in $E$, such
that $\ol{\rho}_w(\Gamma)$ is a big subgroup of $\GL_n(\F_{\ell})$ for any $w \in L$ lying over a prime $\ell \in P$.
\end{theorem}

\begin{proof}
Let $G_w$ be the Zariski closure of $\rho_w(\Gamma)$ in $\GL_n(E_w)$.  Let $P_0$ be the set of primes provided by
Proposition~\ref{compsys-1}.  Then as $w$ varies amongst places of $L$ lying over elements of $P_0$ the index of
$G_w^{\circ}$ in $G_w$ is bounded.  Thus by Proposition~\ref{nhyper-5}, $\ol{\rho}_w(\Gamma)$ is a big subgroup of
$\GL_n(\F_{\ell})$ if $w \in L$ lies over $\ell \in P_0$ and $\ell$ is sufficiently large.  It follows that we can
take $P$ to be the set of all sufficiently large elements of $P_0$.
\end{proof}

We expect that one should be able to take the set $P$ of primes in the above theorem to have density one, but we have
not proved this.  Applying the theorem in the case where $\Gamma$ is the absolute Galois group of a number field and
$E=\Q$ gives Theorem~\ref{mainthm} from the introduction.

\end{document}